\documentclass[12pt]{amsart}
\usepackage[all]{xy}
\usepackage{a4wide}
\usepackage{amssymb}
\usepackage{caption}
\usepackage{amsthm}
\usepackage{amsmath}
\usepackage{amscd,enumitem}
\usepackage{verbatim}
\usepackage{float}
\usepackage{color}
\usepackage[all]{xy}
\usepackage{hyperref}
\usepackage{cleveref}
\theoremstyle{plain}
\newtheorem{theorem}{Theorem}
\numberwithin{theorem}{section}
\newtheorem{corollary}[theorem]{Corollary}

\newtheorem{proposition}[theorem]{Proposition}

\theoremstyle{definition}
\newtheorem{definition}[theorem]{Definition}
\newtheorem{example}{Example}

\newtheorem*{theorem*}{Theorem}
\theoremstyle{remark}
\newtheorem*{remark}{Remark}

\begin{document}

\title{Partition-theoretic Frobenius-type limit formulas}

\author{Robert Schneider}

\address{Department of Mathematical Sciences \newline
Michigan Technological University\newline
Houghton, Michigan 49931, U.S.A.}
\email{robertsc@mtu.edu}
\begin{abstract} 

%RJUPDATE VERSION 5 semicolon instead of period in middle of abstract 

%In recent papers following up on work of Alladi, the author and his collaborators proved numerous partition-theoretic formulas to compute arithmetic densities of subsets of natural numbers, as limiting values of $q$-series as $q\to \zeta$ a root of unity%. As a generalization of that work, u
%Using a formula of Frobenius that generalizes Abel's convergence theorem for power series, filtered through partition generating function techniques, w

Using partition generating function techniques, we prove $q$-series analogues of a formula of Frobenius generalizing Abel's convergence theorem for complex power series. Frobenius' result  states that for $|q|<1$, $\lim_{q\to 1}(1-q)\sum_{n\geq 1} f(n) q^n $ is equal to the average value $\lim_{N\to \infty}$ $\frac{1}{N}\sum_{k=1}^{N}f(k)$ of the sequence $\{f(n)\}$ as $n\to \infty$, if the average value exists.

%Frobenius' Abelian formula. 
%As an application we give a limiting formula for the $q$-bracket of Bloch and Okounkov, an operator from statistical physics connected to the theory of modular forms, as $q\to 1^-$. % from within the unit disk. %, and sum-of-tails and integration formulas. 
\end{abstract}
\maketitle

\section{Introduction and statement of results}\label{Sect1}
%
%%Abelian theorems give alternative ways to compute the limits of convergent infinite series and sequences; by contrast, 
%Abelian theorems give alternative ways to assign limiting values to sequences and series that may not satisfy the usual convergence criteria. For instance, for an arithmetic function $f:\mathbb N \to \mathbb C$ (for which we define $f(0):=0$), if the limit 
%$L:=\lim_{N\to \infty}\  \frac{1}{N}\sum_{1\leq k \leq N}f(k)$ exists then for $q\in\mathbb C, |q|<1$, a theorem of Frobenius \cite{Frobenius} says when $q\to 1$ {radially} %(through real values) 
%we have
%
%
%
In \cite{Abel}, Abel proves a foundational theorem on the convergence of complex power series.

\begin{proposition}[Abel's Convergence Theorem]\label{Abel}%
Let $f\colon \mathbb N \to \mathbb C$ be an arithmetic function. For $q\in \mathbb C, |q|<1$, if the limit 
$L=\lim_{N\to \infty}\sum_{1\leq k \leq N}f(k)$ exists, then 
\begin{equation*}%\label{Frob}
\lim_{q\to 1}\  \sum_{n\geq 1}f(n)q^n\   =\  L
\end{equation*}
as $q\to 1$ radially from within the unit disk.
%\begin{equation}\label{e(Q, 1)}
%f_{avg}(N):=\frac{1}{N}\sum_{1\leq k \leq N}f(k) \  \  \text{with}\  \   \lim_{N\to \infty} f_{avg}(N)=L<\infty,
%\end{equation}
%then for $|q|<1$, when $q\to 1^-$ radially we have
%\begin{equation}\label{Frob}
%\lim_{q\to 1^-}(1-q)\sum_{n\geq 1}f(n)q^n =L.
%\end{equation}
%as $q\to 1$ {radially}; see also \cite{Tauberian}.
\end{proposition}

Another ``Abel type'' theorem giving limiting values  as $q\to 1$ for certain classes of complex power series, is proved by Frobenius in \cite{Frobenius}.\footnote{Prop. \ref{Frob} is an equivalent statement   to the second equation of \cite{Tauberian}, replacing $a_n$ by $f(n)$, and $A$ by $f_{\operatorname{avg}}$; the condition that $f_{\operatorname{avg}}$ exists is equivalent to the Tauberian condition $\sum_{k\leq n}a_k\sim An$.}  %(through real values) 
  % that for a function $f\colon \mathbb N \to \mathbb C$, % (for which we will set $f(0):=0$ in this study), 
%if the average value 
%$f_{\operatorname{avg}}:=\lim_{N\to \infty}$ $\frac{1}{N}\sum_{k=1}^{N}f(k)$ exists, then  $f_{\operatorname{avg}}$ can be computed using  the following formula.%we have
%
\begin{proposition}[Frobenius' Theorem]\label{Frob}%
Let $f\colon \mathbb N \to \mathbb C$ be an arithmetic function. For $q\in \mathbb C, |q|<1$, if the average value 
$f_{\operatorname{avg}}:=\lim_{N\to \infty}$ $\frac{1}{N}\sum_{k=1}^{N}f(k)$ exists, then 
\begin{equation*}%\label{Frob}
\lim_{q\to 1}\  (1-q)\sum_{n\geq 1}f(n)q^n\   =\  f_{\operatorname{avg}}
\end{equation*}
as $q\to 1$ radially from within the unit disk.
%\begin{equation}\label{e(Q, 1)}
%f_{avg}(N):=\frac{1}{N}\sum_{1\leq k \leq N}f(k) \  \  \text{with}\  \   \lim_{N\to \infty} f_{avg}(N)=L<\infty,
%\end{equation}
%then for $|q|<1$, when $q\to 1^-$ radially we have
%\begin{equation}\label{Frob}
%\lim_{q\to 1^-}(1-q)\sum_{n\geq 1}f(n)q^n =L.
%\end{equation}
%as $q\to 1$ {radially}; see also \cite{Tauberian}.
\end{proposition}

%Proposition \ref{Frob} reduces to Proposition \ref{Abel} when the sequence $\{f(n)\}$ converges, since $f_{\operatorname{avg}}=f_{\operatorname{sum}}$. %If $\{f(n)\}$ diverges, %then by the theory of divergent series \cite{Divergent} 
%one may assign a limiting value to the sequence, e.g. the (C, 1) limit in the terminology of \cite{Divergent}, if the average value $f_{\operatorname{avg}}$ exists. % (see \cite{Divergent}). %; in the terminology of G. H. Hardy \cite{Divergent}, one would say $\sum_n f(n)=f_{\operatorname{avg}} \  (C)$.\footnote{See \cite{Divergent, Tauberian} for more about Abelian theorems and the complementary class of Tauberian theorems.}  
%For $f(n)\geq 0$, the converse was proved by Hardy-Littlewood. 
%These types of formulas are classified colloquially as ``Abelian and Tauberian theorems''. 
%	RJUPDATE VERSION 5 methods from q-series --delete "and partition thoery" --- insert "partition-theoretic after italics
In this paper, we prove theorems analogous to Proposition \ref{Frob} using methods from $q$-series, that we will refer to as  {\it Frobenius-type limit formulas}. We note that the limit in Proposition \ref{Frob} holds if $q\to 1$ through any path in a {\it Stolz sector} of the unit disk, a region with vertex at $z=1$ such that $\frac{|1-q|}{1-|q|}\leq M$ for some $M> 0$ (see \cite{Stolz}).  %\begin{remark}
%, roughly complementary classes. 
%\end{remark}

%\end{remark}
Let $\mathcal P$ denote the {\it integer partitions} \cite{And}. For $\lambda \in \mathcal P$, let $|\lambda|$ denote the {\it size} of $\lambda$ (sum of parts), $\ell(\lambda)$ denote the {\it length} (number of parts), and let $\operatorname{sm}(\lambda)$ and $\operatorname{lg}(\lambda)$ denote the {\it smallest part }and {\it largest part} of $\lambda$, respectively, noting $|\emptyset|=\ell(\emptyset)= \operatorname{sm}(\emptyset)= \operatorname{lg}(\emptyset):=0$ for $\lambda=\emptyset$ the empty partition. %We write ``$\lambda \vdash n$'' to mean $|\lambda|=n\geq 0$. %We take $|q|<1$ throughout. 
For $z,q\in \mathbb C, |zq|<1$, let $(z;q)_n:=\prod_{0\leq k <n}(1-zq^k)$ denote the {\it $q$-Pochhammer symbol}, with $(z;q)_{\infty}:=\lim_{n\to \infty} (z;q)_{n}$. Let $p(n)=\sum_{|\lambda|=n}1$ denote the {\it partition function} (number of partitions of size $n\geq 0$), with the initial value $p(0):=1$.%With these notations, we can state our partition-theoretic Tauberian theorems. %, of partition $\lambda$. 
 %, $L:=\lim_{N\to \infty}\frac{1}{N}\sum_{1\leq k \leq N}f(k)$. %, as above. %
%\footnote{For further reading on partition theory, see \cite{And}.}

Note that if $f(n)$ is the indicator function for a subset $S\subseteq \mathbb N$ with arithmetic density $d_S$, then Proposition \ref{Frob} gives the limiting value $f_{\operatorname{avg}}=d_S$ as $q\to 1$. Inspired by work of Alladi \cite{A}, in  \cite{Paper1, Paper2, SS_density}, %inspired by work of Alladi \cite{A}, 
the author and my collaborators 
exploited this idea to prove partition-theoretic and $q$-series formulas for $d_S$ with $q\to 1$, as well as at other roots of unity $\zeta$. %\footnote{Indeed, the vertex of the Stolz sector could be moved to other roots of unity.} 
%In this paper, 
%RJUPDATE VERSION 5 , which extend results of \cite{A} BELOW BEFORE SEMICOLOM
The present note is a complement to the papers \cite{Paper1, Paper2}; we give a general setting in which such partition-theoretic density computations arise naturally. 
{\it Throughout this paper, % formulas, % we take $f$ to be an arithmetic function with $f(0):=0$, and set $L:=\lim_{N\to \infty}\frac{1}{N}\sum_{1\leq k \leq N}f(k)$ if the limit exists, as above. W
we take $q\to 1$ in a Stolz sector of the unit disk.} %For full generality in our statements, we define the following concept.%\footnote{Generalizing $q$-commensurate subsets $S\subseteq \mathbb N$ in \cite{Paper2}, which are  the case $f$ equals the indicator function of $S$.}

%
%\begin{definition}
%We define an arithmetic function $f(n)$ to be {\it $q$-commensurate} if and only if the limit identity Proposition \ref{Frob} continues to hold when $q\to 1$ through an arbitrary path in the unit disk.%, and $L=\lim_{N\to \infty}\frac{1}{N}\sum_{1\leq k \leq N}f(k)<\infty$. 
%\end{definition}
%
%\begin{remark}
%By Proposition \ref{Frob}, one at least has that every arithmetic function $f$ is ``radially $q$-commensurate''.

It is not hard to write down partition-theoretic analogues of Proposition \ref{Frob}. Noting that $f(n)q^n  =\  \frac{f(n)}{p(n)}q^n\cdot \sum_{|\lambda|= n}1\   =\  \sum_{|\lambda|= n} \frac{f(|\lambda|)}{p(|\lambda|)}q^{|\lambda|},$ then Proposition \ref{Frob} can be rewritten as a sum over partitions:
\begin{equation}\label{Frob_P}
\lim_{q\to 1}\  (1-q)\sum_{\lambda \in \mathcal P} \frac{f\left(|\lambda|\right)}{p\left(|\lambda|\right)}q^{|\lambda|}\  \   =\  \  f_{\operatorname{avg}}.
\end{equation}
%\nolinebreak
This resembles Proposition \ref{Frob} somewhat in form, but writing down the coefficients explicitly requires one to repeatedly compute the partition function, a nontrivial task.  %\footnote{Eq. \eqref{Frob_P} does give intuitive information about convergence of infinite sums over partitions, noting that $p(n)\sim \frac{1}{4n\sqrt{3}}\operatorname{exp}\left(\pi\sqrt{2n/3}\right)$ grows very rapidly in the denominator as $n=|\lambda|$ grows large: {coefficients in partition-indexed $q$-series must look significantly smaller than those of conventional power series, in order to yield comparable convergence behaviors.}} 
Alternatively, replacing $f(n)$ in Proposition \ref{Frob} with $(f\cdot p)(n):=f(n)p(n)$ gives by the same argument
\begin{equation}\label{Frob_P2}
\lim_{q\to 1}\  (1-q)\sum_{\lambda \in \mathcal P}f\left(|\lambda|\right) q^{|\lambda|}\  \   =\  \  {(f\cdot p)}_{\operatorname{avg}},
\end{equation}
%where the sum is taken over all partitions $\lambda$. 
%This resembles Proposition \ref{Frob} to some extent, and may be useful if
if  ${(f\cdot p)}_{\operatorname{avg}}=\lim_{N\to \infty}\frac{1}{N}\sum_{k=1}^{N}f(k)p(k)$ exists. However, neither \eqref{Frob_P} nor \eqref{Frob_P2} strongly resembles Proposition \ref{Frob}, in that the limits  on the right-hand sides do not equal the average values of the coefficients $f(|\lambda|)$ on the left, but versions weighted by $p(n)$.  %Indeed, we will see that  %In fact, a multiplicative factor that converges to zero much faster than $1-q$ is needed to counteract the divergence of the sum over partitions as $q\to 1$, and 
%size $|\lambda|$ is the wrong  partition statistic to compose with $f(n)$ for the ``right'' analogues of Proposition \ref{Frob}. %Here is the correct partition analogue of Proposition \ref{Frob}.

Below we prove a number of Frobenius-type limit formulas that represent more faithful analogues of Proposition \ref{Frob}. The proofs of these formulas hold for arithmetic functions $f(n)$ that we will refer to as having the property of ``$q$-summability''.\footnote{We do not prove general $q$-summability theorems here. The property must be checked for a given $f(n)$; general proofs of $q$-summability would be useful. We note here, as remarks, examples from previous works \cite{Paper1, Paper2, SchneiderPhD} proved by less general methods, as demonstrations that our general limit theorems are not vacuous.} % -- but rather depend on the existence of examples of each type already proved in \cite{Paper1, Paper2, SchneiderPhD} as proofs of concept.}

\begin{definition}\label{def}
Suppose for arithmetic function $f\colon \mathbb N \to \mathbb C$ that the limit $f_{\operatorname{avg}}:=\lim_{N\to \infty}$ $\frac{1}{N}\sum_{k=1}^{N}f(k)$ exists. For $|q|<1$, write
\begin{equation}\label{defeq} (1-q)\sum_{n\geq 1}f(n)q^n\   =\  f_{\operatorname{avg}}q\  +\  \varepsilon_f(q)q, \end{equation}
noting by Proposition \ref{Frob} that as $q\to 1$, the {error function} $\varepsilon_f(q)\to 0$. 

We define $f(n)$ to be a {\it $q$-summable function of type} (Q, 1) if $\sum_{k\geq 1}f(k)q^k(q;q)_k^{-1}$ is absolutely convergent, and  the following condition  holds: 
\begin{equation}\label{Q1def} \lim_{q\to 1}\sum_{k\geq 1}\frac{\varepsilon_f(q^k)q^{k}}{(q;q)_{k}}=0.\end{equation}
We define $f(n)$ to be a {\it $q$-summable  function of type} (Q, 2) if $\sum_{k\geq 1}f(k)q^k(q;q)_{k-1}$ is absolutely convergent, and the following condition holds:
%\begin{enumerate} 
\begin{equation}\label{Q2def} \lim_{q\to 1}\sum_{k\geq 1}\frac{(-1)^{k+1}\  \varepsilon_f(q^k)q^{\frac{k(k+1)}{2}}}{(q;q)_{k}}=0.\end{equation}
%Furthermore, following Hardy \cite{Divergent},  if $f(n)$ is  a  $q$-summable function of type (Q, $i$), $i\in \{1,2\}$, we say $\sum_{n\geq 1} f(n)$ is summable (Q, $i$), and say $f_{\operatorname{avg}}$ is the (Q, $i$) limit of the sequence $\{f(n)\}$.
%\  \\
%\item[(Q, 2)] $$\lim_{q\to 1}\sum_{k\geq 1}\frac{\varepsilon_f(q^k)q^{k}}{(q;q)_{k}}=0.$$
%\end{enumerate}
 \end{definition}

%
%\begin{remark}
%Following Hardy \cite{Divergent},  if $f(n)$ is a  $q$-summable function of type (Q, $i$), $i\in \{1,2\}$, one can %might %say $\sum_{n\geq 1} f(n)$ is summable (Q, $i$), and 
%say $f_{\operatorname{avg}}$ is the (Q, $i$) limit of the sequence $\{f(n)\}$.\end{remark}
%
%\begin{remark} By the comparison test, property (Q, 2) implies property (Q, 1) holds as well.
%\end{remark}

\begin{remark}
The property of $q$-summability generalizes the idea of $q$-commensurate subsets of $\mathbb N$ in \cite{Paper2}: $S\subseteq \mathbb N$ is $q$-commensurate if and only if the indicator function of $S$ is $q$-summable.
\end{remark}

\begin{remark}
We loosely imitate the notations for summation methods (C, 1), (C, 2), etc. in \cite{Divergent}. 
\end{remark}

\begin{theorem}\label{thm2}
For $f(n)$ a $q$-summable arithmetic function of type (Q, 1), if the limit $f_{\operatorname{avg}}=\lim_{N\to \infty}\frac{1}{N}\sum_{k=1}^{N} f(k)$ exists, then 
\begin{flalign*}
%\lim_{N\to \infty}\frac{1}{\log N}{\sum_{n \in \mathcal S(N)}\frac{1}{n}}=
 \lim_{q\to 1}\  (q;q)_{\infty}\sum_{\lambda \in \mathcal P} f\left(\operatorname{sm}(\lambda)\right)q^{|\lambda|}\  %=\  \lim_{q\to 1}\  (q;q)_{\infty}\sum_{n\geq 1}\frac{f(n)q^n}{(q;q)_n}\  %=\   -\lim_{q\to 1}\  \sum_{\lambda \neq \emptyset }\mu_{\mathcal P}(\lambda)f\left(\operatorname{sm}(\lambda)\right)q^{|\lambda|}\  
=\  f_{\operatorname{avg}}, %,\\
%\lim_{q\to 1}\  (q;q)_{\infty}\sum_{n\geq 1}\frac{f(n)q^n}{(q;q)_n}\  =\  \lim_{q\to 1}\  \sum_{\lambda \in \mathcal P}\mu_{\mathcal P}^*(\lambda)f\left(\operatorname{sm}(\lambda)\right)q^{|\lambda|}\  =\  L.
\end{flalign*}
where the sum is taken over all partitions, and $\operatorname{sm}(\lambda)$ denotes the smallest part of $\lambda\in \mathcal P$. \end{theorem}
%where the sum is taken over all partitions $\lambda$.

%
%
%
%
%*****
%
%\begin{theorem}\label{thm1}
%For $f(n)$ a $q$-summable arithmetic function of type (Q, 2), if the limit $f_{\operatorname{avg}}=\lim_{N\to \infty}\frac{1}{N}\sum_{k=1}^{N} f(k)$ exists, then \begin{flalign*}
%%\lim_{N\to \infty}\frac{1}{\log N}{\sum_{n \in \mathcal S(N)}\frac{1}{n}}=
% \lim_{q\to 1}\  (q;q)_{\infty}\sum_{\lambda \in \mathcal P} f\left(\operatorname{lg}(\lambda)\right)q^{|\lambda|}\  %=\  \lim_{q\to 1}\  (q;q)_{\infty}\sum_{n\geq 1}\frac{f(n)q^n}{(q;q)_n}\  %=\   -\lim_{q\to 1}\  \sum_{\lambda \neq \emptyset }\mu_{\mathcal P}(\lambda)f\left(\operatorname{sm}(\lambda)\right)q^{|\lambda|}\  
%=\  f_{\operatorname{avg}}, %,\\
%%\lim_{q\to 1}\  (q;q)_{\infty}\sum_{n\geq 1}\frac{f(n)q^n}{(q;q)_n}\  =\  \lim_{q\to 1}\  \sum_{\lambda \in \mathcal P}\mu_{\mathcal P}^*(\lambda)f\left(\operatorname{sm}(\lambda)\right)q^{|\lambda|}\  =\  L.
%\end{flalign*}
%where the sum is taken over all partitions, and $\operatorname{lg}(\lambda)$ denotes the largest part of  $\lambda\in \mathcal P$. \end{theorem}
%

We prove this theorem and all other results in Section \ref{Sect2} below.
Theorem \ref{thm2} is a true partition-theoretic analogue of Frobenius' formula in Proposition \ref{Frob}. %; if $f(n)\to f_{\operatorname{avg}}$ as $n\to \infty$, then the theorem yields a partition analogue of Abel's convergence theorem. 

Partition generating function methods lead to further formulas to compute the limit $f_{\operatorname{avg}}$. 
We require the {partition-theoretic M\"{o}bius function} $\mu\colon \mathcal P \to \{-1, 0, 1\}$ defined in \cite{Schneider_arithmetic}: 
\begin{equation}\label{mudef}
\mu_{\mathcal{P}}(\lambda) := \begin{cases} 0 & \rm{if} \ \lambda \ \rm{has \ any\  part\  repeated}, \\
(-1)^{\ell(\lambda)} & \rm{otherwise}. \end{cases}  
\end{equation}

\begin{corollary}\label{cor2}
For $f(n)$ a $q$-summable arithmetic function of type (Q, 1), if the limit $f_{\operatorname{avg}}=\lim_{N\to \infty}\frac{1}{N}\sum_{k=1}^{N} f(k)$ exists, then 
\begin{flalign*}
%\lim_{N\to \infty}\frac{1}{\log N}{\sum_{n \in \mathcal S(N)}\frac{1}{n}}=
-\lim_{q\to 1}\   \sum_{\lambda \in \mathcal P}\mu_{\mathcal P}(\lambda)f\left(\operatorname{lg}(\lambda)\right)q^{|\lambda|} \  =\  f_{\operatorname{avg}}. %,\\
%\lim_{q\to 1}\  (q;q)_{\infty}\sum_{n\geq 1}\frac{f(n)q^n}{(q;q)_n}\  =\  \lim_{q\to 1}\  \sum_{\lambda \in \mathcal P}\mu_{\mathcal P}^*(\lambda)f\left(\operatorname{sm}(\lambda)\right)q^{|\lambda|}\  =\  L.
\end{flalign*}
\end{corollary}

\begin{corollary}\label{cor2.5}
For $f(n)$ a $q$-summable arithmetic function of type (Q, 1), if the limit $f_{\operatorname{avg}}=\lim_{N\to \infty}\frac{1}{N}\sum_{k=1}^{N} f(k)$ exists, then 
\begin{flalign*}
%\lim_{N\to \infty}\frac{1}{\log N}{\sum_{n \in \mathcal S(N)}\frac{1}{n}}=
\lim_{q\to 1}\ \sum_{n\geq 1}f(n)q^n(q;q)_{n-1}\  = \  \lim_{q\to 1}\ \sum_{n\geq 1}\sum_{k\geq 1}\frac{f(n)q^{nk}}{(q;q)_{k-1}}\  =\  f_{\operatorname{avg}}. %,\\
%\lim_{q\to 1}\  (q;q)_{\infty}\sum_{n\geq 1}\frac{f(n)q^n}{(q;q)_n}\  =\  \lim_{q\to 1}\  \sum_{\lambda \in \mathcal P}\mu_{\mathcal P}^*(\lambda)f\left(\operatorname{sm}(\lambda)\right)q^{|\lambda|}\  =\  L.
\end{flalign*}
\end{corollary}

\begin{remark}
Setting $f(n)$ equal to the indicator function for $S\subseteq \mathbb N$, then that the first limit in Corollary \ref{cor2.5} is equal to $f_{\operatorname{avg}}=d_{S}$, re-proves Theorem 3.6 of \cite{Paper2}.
\end{remark}

Somewhat surprisingly, if one replaces ``$\operatorname{sm}$'' with ``$\operatorname{lg}$'' in Theorem \ref{thm2}, the limit still holds.

\begin{theorem}\label{thm1}
For $f(n)$ a $q$-summable arithmetic function of type (Q, 2), if the limit $f_{\operatorname{avg}}=\lim_{N\to \infty}\frac{1}{N}\sum_{k=1}^{N} f(k)$ exists, then \begin{flalign*}
%\lim_{N\to \infty}\frac{1}{\log N}{\sum_{n \in \mathcal S(N)}\frac{1}{n}}=
 \lim_{q\to 1}\  (q;q)_{\infty}\sum_{\lambda \in \mathcal P} f\left(\operatorname{lg}(\lambda)\right)q^{|\lambda|}\  %=\  \lim_{q\to 1}\  (q;q)_{\infty}\sum_{n\geq 1}\frac{f(n)q^n}{(q;q)_n}\  %=\   -\lim_{q\to 1}\  \sum_{\lambda \neq \emptyset }\mu_{\mathcal P}(\lambda)f\left(\operatorname{sm}(\lambda)\right)q^{|\lambda|}\  
=\  f_{\operatorname{avg}}, %,\\
%\lim_{q\to 1}\  (q;q)_{\infty}\sum_{n\geq 1}\frac{f(n)q^n}{(q;q)_n}\  =\  \lim_{q\to 1}\  \sum_{\lambda \in \mathcal P}\mu_{\mathcal P}^*(\lambda)f\left(\operatorname{sm}(\lambda)\right)q^{|\lambda|}\  =\  L.
\end{flalign*}
where the sum is taken over all partitions, and $\operatorname{lg}(\lambda)$ denotes the largest part of  $\lambda\in \mathcal P$. \end{theorem}

%
%\begin{theorem}\label{thm2}
%For $f(n)$ a $q$-summable arithmetic function of type (Q, 1), if the limit $f_{\operatorname{avg}}=\lim_{N\to \infty}\frac{1}{N}\sum_{k=1}^{N} f(k)$ exists, then 
%\begin{flalign*}
%%\lim_{N\to \infty}\frac{1}{\log N}{\sum_{n \in \mathcal S(N)}\frac{1}{n}}=
% \lim_{q\to 1}\  (q;q)_{\infty}\sum_{\lambda \in \mathcal P} f\left(\operatorname{sm}(\lambda)\right)q^{|\lambda|}\  %=\  \lim_{q\to 1}\  (q;q)_{\infty}\sum_{n\geq 1}\frac{f(n)q^n}{(q;q)_n}\  %=\   -\lim_{q\to 1}\  \sum_{\lambda \neq \emptyset }\mu_{\mathcal P}(\lambda)f\left(\operatorname{sm}(\lambda)\right)q^{|\lambda|}\  
%=\  f_{\operatorname{avg}}, %,\\
%%\lim_{q\to 1}\  (q;q)_{\infty}\sum_{n\geq 1}\frac{f(n)q^n}{(q;q)_n}\  =\  \lim_{q\to 1}\  \sum_{\lambda \in \mathcal P}\mu_{\mathcal P}^*(\lambda)f\left(\operatorname{sm}(\lambda)\right)q^{|\lambda|}\  =\  L.
%\end{flalign*}
%where the sum is taken over all partitions, and $\operatorname{sm}(\lambda)$ denotes the smallest part of $\lambda\in \mathcal P$. \end{theorem}
%%where the sum is taken over all partitions $\lambda$.
%
%
%
%
%
%

Theorem \ref{thm1} is a second partition analogue of Proposition \ref{Frob}. As with Theorem \ref{thm2}, generating function methods yield further formulas to compute $f_{\operatorname{avg}}$.

\begin{corollary}\label{cor1}
For $f(n)$ a $q$-summable arithmetic function of type (Q, 2), if the limit $f_{\operatorname{avg}}=\lim_{N\to \infty}\frac{1}{N}\sum_{k=1}^{N} f(k)$ exists, then 
\begin{flalign*}
%\lim_{N\to \infty}\frac{1}{\log N}{\sum_{n \in \mathcal S(N)}\frac{1}{n}}=
 -\lim_{q\to 1}\   \sum_{\lambda \in \mathcal P}\mu_{\mathcal P}(\lambda)f\left(\operatorname{sm}(\lambda)\right)q^{|\lambda|}\   =\  f_{\operatorname{avg}}. %,\\
%\lim_{q\to 1}\  (q;q)_{\infty}\sum_{n\geq 1}\frac{f(n)q^n}{(q;q)_n}\  =\  \lim_{q\to 1}\  \sum_{\lambda \in \mathcal P}\mu_{\mathcal P}^*(\lambda)f\left(\operatorname{sm}(\lambda)\right)q^{|\lambda|}\  =\  L.
\end{flalign*}
\end{corollary}

\begin{corollary}\label{cor1.5}
For $f(n)$ a $q$-summable arithmetic function of type (Q, 2), if the limit $f_{\operatorname{avg}}=\lim_{N\to \infty}\frac{1}{N}\sum_{k=1}^{N} f(k)$ exists, then 
\begin{flalign*}
%\lim_{N\to \infty}\frac{1}{\log N}{\sum_{n \in \mathcal S(N)}\frac{1}{n}}=
\lim_{q\to 1}\ (q;q)_{\infty}\sum_{n\geq 1}\frac{f(n)q^n}{(q;q)_{n}}\  =\  -\lim_{q\to 1}\  \sum_{n\geq 1}\sum_{k\geq 1}\frac{(-1)^k f(n) q^{nk+\frac{k(k-1)}{2}}}{(q;q)_{k-1}}\  =\  f_{\operatorname{avg}}. %,\\
%\lim_{q\to 1}\  (q;q)_{\infty}\sum_{n\geq 1}\frac{f(n)q^n}{(q;q)_n}\  =\  \lim_{q\to 1}\  \sum_{\lambda \in \mathcal P}\mu_{\mathcal P}^*(\lambda)f\left(\operatorname{sm}(\lambda)\right)q^{|\lambda|}\  =\  L.
\end{flalign*}
\end{corollary}

\begin{remark}
Setting $f(n)$ equal to the indicator function for $S\subseteq \mathbb N$ in Theorem \ref{thm1} and Corollaries \ref{cor1} and \ref{cor1.5}, re-proves the main arithmetic density results of \cite{Paper1, Paper2} for the case $\zeta=1$.
\end{remark}

\begin{remark}
%Replacing $f(n)$ with $\frac{f(n)}{n}$, then t
That the left-hand limit in Corollary \ref{cor1.5} is equal to $f_{\operatorname{avg}}$ confirms the conjecture the author made just below Example E.1.1 in \cite{SchneiderPhD} for the case $\zeta=1$.
\end{remark}

\begin{example}
Set $f(n)=\frac{\varphi(n)}{n}$ in Corollary \ref{cor1.5} with $\varphi(n)$ the Euler phi function; it is a well-known result that $\frac{1}{N}\sum_{k=1}^{N}f(k) \sim 6/\pi^2$ as $N\to \infty$. Noting $f(n)$ is $q$-summable of type (Q, 2) (see remark below), then %we have 
$$\lim_{q\to 1}\ (q;q)_{\infty}\sum_{n\geq 1}\frac{\varphi(n)q^n}{n\cdot (q;q)_{n}} \  =\  \frac{6}{\pi^2}.
$$
\end{example}

\begin{remark}
This re-proves Example E.1.1 of \cite{SchneiderPhD} for the case $\zeta=1$. 
\end{remark}

\begin{remark}
One anticipates similar limiting formulas hold as $q$ approaches other roots of unity.
\end{remark}

%
%
%\begin{corollary}
%If $f(x):=\sum_{0\leq k<x}\phi(k)$ with $\phi(n)$ an arithmetic function defined at $n=0$, then we have %such that $L:=\lim_{T\to \infty}\frac{1}{T}\int_{1}^{T}f(x)dx.$ Then we have
%\begin{flalign*}
%%\lim_{N\to \infty}\frac{1}{\log N}{\sum_{n \in \mathcal S(N)}\frac{1}{n}}=
%-\lim_{q\to 1}\ (q;q)_{\infty} \int_{1}^{\infty}f(x)\left[ \frac{d}{dx}(q^{x+1};q)_{\infty}^{-1}\right] dx \  =\  L. %,\\
%%\lim_{q\to 1}\  (q;q)_{\infty}\sum_{n\geq 1}\frac{f(n)q^n}{(q;q)_n}\  =\  \lim_{q\to 1}\  \sum_{\lambda \in \mathcal P}\mu_{\mathcal P}^*(\lambda)f\left(\operatorname{sm}(\lambda)\right)q^{|\lambda|}\  =\  L.
%\end{flalign*}
%\end{corollary}
%

\section{Proofs of results}\label{Sect2}

The proofs in this section begin with a rewriting of equation \eqref{defeq} from Definition \eqref{def}: %, multiplied by an additional factor of $q$ (which does not affect the asymptotic below), as $q\to1$ in a Stolz sector of the unit disk:
\begin{equation}\label{qasymp}
\sum_{n\geq 1}f(n)q^n \  =\   f_{\operatorname{avg}}\cdot \frac{q}{1-q}\  +\   \frac{\varepsilon_f(q)q}{1-q}, %\  \sim\  f_{\operatorname{avg}}\cdot \frac{q}{1-q},
\end{equation}
with $f_{\operatorname{avg}}:=\lim_{N\to \infty}\  \frac{1}{N}\sum_{k=1}^{N}f(k)$, as above, and $\varepsilon_f(q)$ as defined by  \eqref{defeq}.
%Our proof requires that all implicit error terms in the following asymptotic relations are of lower order than the main terms; then the error terms will ultimately vanish when multiplied by $(q;q)_{\infty}$ as $q\to 1$ in the final steps of the proofs, whereas the main terms will become finite under these conditions. Thus we proceed by using asymptotics ``$\sim$'', keeping in mind this assumption about errors. 

We use equation \eqref{qasymp}  as a building block to produce further $q$-series formulas to compute the limit $f_{\operatorname{avg}}$. 
The proof of Theorem \ref{thm2} below generalizes the proof of Theorem 3.6 in \cite{Paper2}.

\begin{proof}[Proof of Theorem \ref{thm2}]
Take $q\mapsto q^k$ in \eqref{qasymp}. Multiply both sides by $(q;q)_{k-1}^{-1}$, sum over $k\geq 1$, then swap order of summation of the double summation, and make the change of indices $k\mapsto k+1$ on the left, to give %and take $k\mapsto k+1$ in the indices on the left to give
\begin{flalign}\label{asymp3}
\sum_{n\geq 1}\sum_{k\geq 1}\frac{f(n)q^{nk}}{(q;q)_{k-1}}&=\sum_{n\geq 1}f(n)q^n\sum_{k\geq 0}\frac{q^{nk}}{(q;q)_{k}}\\ \nonumber &=\  f_{\operatorname{avg}} \cdot \sum_{k\geq 1}\frac{q^k}{(q;q)_k}\  +\  \sum_{k\geq 1}\frac{\varepsilon_f(q^k)q^{k}}{(q;q)_{k}}\  \sim\  f_{\operatorname{avg}} \cdot \sum_{k\geq 1}\frac{q^k}{(q;q)_k}
\end{flalign}
%The sum on the left generates partitions with largest part having multiplicity in $\mathcal S$ (which conjugates to the set of partitions with smallest part in $\mathcal S$). 
as $q\to 1$. We note that both the asymptotic, and the order-of-summation swap by absolute convergence, are justified by the hypothesis that $f$ is $q$-summable of type (Q, 1) (see \eqref{Q1def}). 

By the $q$-binomial theorem \cite{And}, the inner sum over $k\geq 0$ in the second double series is equal to $(q^n;q)_{\infty}^{-1}$, and the summation on the right is $(q;q)_{\infty}^{-1}-1$. Multiplying both sides of \eqref{asymp3} by $(q;q)_{\infty}$ gives, from standard generating function arguments, % implies, % by well-known generating function arguments, 
as $q\to 1$:
 \begin{flalign}\label{asymp4}
(q;q)_{\infty}\sum_{\lambda \in \mathcal P} f\left(\operatorname{sm}(\lambda)\right)q^{|\lambda|}=(q;q)_{\infty}\sum_{n\geq 1} \frac{f(n)q^n}{(q^{n};q)_{\infty}}\ &=\  \sum_{n\geq 1}f(n)q^n(q;q)_{n-1}\\ \nonumber &=\  -\sum_{\lambda\in \mathcal P}\mu_{\mathcal P}(\lambda) f\left(\operatorname{lg}(\lambda)\right)q^{|\lambda|}\  \sim \  f_{\operatorname{avg}}.\end{flalign}
That the left-hand side is asymptotically equal to $f_{\operatorname{avg}}$ as $q\to 1$, is equivalent to the statement of the theorem. \end{proof}

\begin{proof}[Proof of Corollaries \ref{cor2} and \ref{cor2.5}]
These corollaries record alternative expressions for the limit $f_{\operatorname{avg}}$ derived during the proof of Theorem \ref{thm2} above. \end{proof}

The following proof of Theorem \ref{thm1} generalizes the proof Theorem 3.5 in \cite{Paper2}. 
%\begin{remark}
%In full generality, a subset $\mathcal S$ being $q$-commensurate is equivalent to one's being able, for that subset, to interchange order of limits in the calculation
%%\begin{equation}\label{ratio2}
%$d_{\mathcal S}: =\lim_{N\to \infty} \lim_{q\to 1} \frac{\sum_{n\in \mathcal S(N)} q^n}{\sum_{0 \leq n < N} q^n}
%$, %\footnote{Swapping limits is highly nontrivial; G. H. Hardy called it ``one of the most important [problems] in mathematics.''\cite{Hardy}.}, 
%{\it and} to having all implicit error terms in the following asymptotic equalities of lower order than the main terms, so the errors vanish when multiplied by $(q;q)_{\infty}$  as $q\to 1$.\end{remark}

\begin{proof}[Proof of Theorem \ref{thm1}]
Take $q\mapsto q^k$ in \eqref{qasymp}.  Multiply through by $(-1)^{k} q^{\frac{k(k-1)}{2}}(q;q)_{k-1}^{-1}$ and sum both sides over $k\geq 1$, swapping order of summation on the left-hand side, to give
\begin{equation}\label{qleft}
\sum_{n \geq 1} \sum_{k\geq 1} \frac{(-1)^{k} f(n) q^{nk+\frac{k(k-1)}{2}}}{(q;q)_{k-1}}.
\end{equation}
For each $k\geq 1$, the factor $(q;q)_{k-1}^{-1}$ generates partitions with largest part strictly $<k$. The factor $q^{nk}$ adjoins a largest part $k$ with multiplicity $n$ to each partition, for every $n\geq 1$. The $q^{k(k-1)/2}=q^{1+2+3+...+(k-1)}$ factor guarantees at least one part of each size $<k$. Thus \eqref{qleft} is the generating function for partitions $\gamma$ with every natural number $<\operatorname{lg}(\gamma)$ appearing as a part, weighted by $(-1)^{\operatorname{lg}(\gamma)}f\left(m_{\operatorname{lg}}(\gamma) \right)=(-1)^k f(n)$, where $m_{\operatorname{lg}}(\gamma)=n$ denotes the multiplicity of the largest part of each partition $\gamma$.

Under conjugation, %consideration of Young diagrams proves 
this set of partitions $\gamma$ maps to partitions $\lambda$ into distinct parts weighted by $\mu_{\mathcal P}(\lambda)f\left(\operatorname{sm}(\lambda)\right)=(-1)^{\ell(\lambda)}f\left(\operatorname{sm}(\lambda)\right)=(-1)^k f(n)$, which is nonzero since $\lambda$ has no repeated part. %Multiplying by $-1$ gives $\mu_{\mathcal P}^*(\lambda):=-\mu_{\mathcal P}(\lambda)$; t
Thus, noting $f(0):=0$, and multiplying through by a factor of $-1$ to produce the desired end result, we have 
\begin{flalign}\label{nexttolast}
-\sum_{n\geq 1}\sum_{k\geq 1}\frac{(-1)^k f(n) q^{nk+\frac{k(k-1)}{2}}}{(q;q)_{k-1}}\  &=\  -\sum_{\lambda \in \mathcal P} \mu_{\mathcal P}(\lambda)f\left(\operatorname{sm}(\lambda)\right)q^{|\lambda|}=\  \sum_{n\geq 1} f(n)q^n(q^{n+1};q)_{\infty}\  \\ \nonumber &= \  (q;q)_{\infty}\sum_{n\geq 1}\frac{f(n)q^n}{(q;q)_n}\  =\  (q;q)_{\infty}\sum_{\lambda \in \mathcal P} f\left(\operatorname{lg}(\lambda)\right)q^{|\lambda|},
\end{flalign}
using standard partition generating function methods. Manipulating the right-hand side of \eqref{qasymp} accordingly gives, as $q\to 1$, 
\begin{flalign}\label{finalstep}
-\sum_{n\geq 1}\sum_{k\geq 1}\frac{(-1)^k f(n)  q^{nk+\frac{k(k-1)}{2}}}{(q;q)_{k-1}}\  &=\   -f_{\operatorname{avg}}\cdot \sum_{k\geq 1}\frac{(-1)^k q^{\frac{k(k+1)}{2}}}{(q;q)_{k}}\  +\  \sum_{k\geq 1}\frac{(-1)^{k+1}\varepsilon_f(q^k)q^{\frac{k(k+1)}{2}}}{(q;q)_{k}}\\  \nonumber &\sim \   -f_{\operatorname{avg}}\cdot \sum_{k\geq 1}\frac{(-1)^k q^{\frac{k(k+1)}{2}}}{(q;q)_{k}}\  =\  -f_{\operatorname{avg}}\cdot \left((q;q)_{\infty}-1 \right).\end{flalign}
The asymptotic and the order-of-summation swaps (by absolute convergence) are justified by the hypothesis that $f$ is $q$-summable of type (Q, 2) (see \eqref{Q2def}); and we use an identity of Euler for the final equality (see \cite{And}), noting the right-hand side approaches $f_{\operatorname{avg}}$ as $q\to 1$. Comparing the right-hand sides of \eqref{nexttolast} and \eqref{finalstep} as $q\to 1$ completes the proof. \end{proof}

\begin{proof}[Proof of Corollaries \ref{cor1} and \ref{cor1.5}]
These two corollaries record alternative expressions for the limit $f_{\operatorname{avg}}$ derived during the proof of Theorem \ref{thm1} above. 
\end{proof}

%
%
%\begin{proof}[Proof of Theorem \ref{cor3}]
%The first identity of the theorem is essentially the special case $f(n)=1$ of Lemma 4.3 in \cite{Paper2}, after adding $a(\emptyset)$ to both sides of the lemma, in light of \cite{Paper2}, eq. (21). Then under the hypotheses in the statement of the theorem, we have 
%\begin{equation}\left<a\right>_q\  =\  (q;q)_{\infty}\sum_{n\geq 0}\frac{A_n(q)q^n}{(q;q)_n}\  \sim\  (q;q)_{\infty}\sum_{n\geq 1}\frac{A_nq^n}{(q;q)_n}\  \sim\  A
%\end{equation}
%as $q\to 1$, noting $a(\emptyset)\cdot (q;q)_{\infty}\to 0$ so the $n=0$ term vanishes in the first asymptotic, and the final asymptotic is the $f(n)=A_n, f_{\operatorname{avg}}=A$ case of Corollary \ref{cor1.5}.
%\end{proof}
%%
%%
%%\begin{remark}
%%If $A(q):=\lim_{N\to \infty}\frac{1}{N}\sum_{1\leq k \leq N}{A}_k(q)$ for $|q|<1$, we conjecture $\left<a\right>_q \sim A(q)$ as $q\to 1$. \end{remark}
%%
%
%
%
%
The identity \eqref{qasymp} could be involved in series manipulations in diverse ways. 
%% much like the formulas involving the $q$-density statistic in \cite{Paper2}, which themselves could be generalized to give new limit formulas as in this paper. For example, if one takes $q\mapsto q^k$ in \eqref{qasymp} and sums over all $k\geq 1$, swapping order of summation on the left in the absolutely convergent double series, then geometric series and classical Lambert series identities\footnote{The reader is referred to \cite{Schmidt} for many useful, closely related Lambert series generating function formulas.} suggest as $q\to 1$ that
%\begin{equation*}
%\lim_{q\to 1}\  \frac{\sum_{n\geq 1}\frac{f(n)q^n}{1-q^n}}{\sum_{n\geq 1}\frac{q^n}{1-q^n}} \  =\  f_{\operatorname{avg}} \end{equation*}
%with limit $f_{\operatorname{avg}}$ as defined previously, a generalization of Example 2 in \cite{Paper2}, which we will not attempt to prove rigorously but does seem believable. %Corollary \ref{cor3} itself is such a limiting formula stemming from \cite{Paper2}.
%
%%One wonders about the existence of even more general limiting formulas of the shape 
%%\begin{equation*}
%%\lim_{q\to 1}\  \frac{\sum_{n\geq 1}f(n)\phi(n)}{\sum_{n\geq 1}\phi(n)} \  =\  f_{\operatorname{avg}} \end{equation*}
%for special $\phi:\mathbb N \to \mathbb C,\  \phi(n)=o(1)$.   
%One anticipates limits like these will hold as $q$ approaches certain other roots of unity, depending in part on the arithmetic function $f(n)$. 
%Certainly, u
Using techniques from $q$-series, modular forms, Lambert series, et al., it seems likely one can produce other classes of limit formulas analogous to Proposition \ref{Frob}. %, given sufficient conditions for convergence. %.\footnote{One must still prove convergence conditions.}

\section*{Acknowledgments}
The author is indebted to George Andrews and Jeffrey Lagarias for discussions about analysis that influenced this paper, and to J. Lagarias for offering useful revisions; to Matthew R. Just, Ken Ono, Paul Pollack, A. V. Sills and Ian Wagner for conversations that advanced my work; and to the anonymous referee for suggestions that strengthened the final draft. %

\end{document}